\documentclass[11pt]{article}
\usepackage
[pdftex
,a4paper
,hyperindex
=
true
,hyperfigures
=
true
,pagebackref
=
true
,bookmarks
=
true
,pdftitle
=
{TITLE
},pdfsubject
=
{},pdfauthor
=
{AUTHOR
},pdfkeywords={KEYWORDS},pdfpagemode=UseOutlines,menubordercolor=1 1 1,colorlinks=true,urlcolor=blue,citecolor=red,linkcolor=blue]{hyperref} 

\usepackage{yfonts}
\usepackage{amsfonts,amsmath,amssymb,epsfig}
\usepackage{amsthm}
\usepackage{graphicx}

\newcommand{\al}{\alpha}

\newcommand{\parx}{\{ x \}}

\newcommand{\BE}{\begin{equation}}
\newcommand{\EE}{\end{equation}}
\newcommand{\la}{\lambda}

\newcommand{\RR}{{\mathbb R}}

\newcommand{\TT}{{\mathbb T}}
\newcommand{\ome}{{\mathbb \omega}}
\newcommand{\R}{{\mathbb R}}

\newcommand{\ZZ}{{\mathbb Z}}
\newcommand{\NN}{{\mathbb N}}
\newcommand{\N}{{\mathbb N}}
\newcommand{\PP}{{\mathbb P}}
\newcommand{\EEE}{{\mathbb E}}

\newcommand{\ep}{\varepsilon}

\newcommand{\pjk}{\Psi^{(i)}_{j,k}}
\newcommand{\ptjk}{{\psi}^{(i)}_{j,k}}
\newcommand{\pjkun}{\psi^{(1)}_{j,k}}

\newtheorem{Theo}{Theorem}[section]
\newtheorem{lem}[Theo]{Lemma}
\newtheorem{coro}[Theo]{Corollary}
\newtheorem{prop}[Theo]{Proposition}
\newtheorem{defi}[Theo]{Definition}
\theoremstyle{definition}
\newtheorem{rem}[Theo]{Remark}
\newcommand{\BP}{\begin{prop}}
\newcommand{\EP}{\end{prop}}
\newcommand{\BC}{\begin{coro}}
\newcommand{\EC}{\end{coro}}
\newcommand{\BL}{\begin{lem}}
\newcommand{\EL}{\end{lem}}
\newcommand{\BD}{\begin{defi}}
\newcommand{\ED}{\end{defi}}
\newcommand{\BT}{\begin{Theo}}
\newcommand{\ET}{\end{Theo}}
\newcommand{\BR}{\begin{rem}}
\newcommand{\ER}{\end{rem}}

\newsavebox{\fmbox}
\newenvironment{fmpage}[1]
 {\begin{lrbox}{\fmbox}\begin{minipage}{#1}}
 {\end{minipage}\end{lrbox}\fbox{\usebox{\fmbox}}}

\def\E{{\hbox{I\kern-.2em\hbox{E}}}}

\graphicspath{{Figures/}}

\author{C\'eline Esser\thanks{Université de Liège, B\^at. B37 Zone Polytech, All\'ee de la D\'ecouverte 12, B-4000 Li\`ege, Belgique},  St\'ephane Jaffard\thanks{ 
        Univ Paris Est Creteil, Univ Gustave Eiffel, CNRS, LAMA UMR8050, F-94010 Creteil, France
jaffard@u-pec.fr } and B\'eatrice
Vedel \thanks{Univ Bretagne Sud, CNRS UMR 6205, LMBA, F-56000, Vannes, France}}

 \date{\today} 
\title{Regularity properties of random wavelet series}

\begin{document}

 \maketitle

\begin{abstract}
We study the regularity properties of random wavelet series constructed by multiplying the coefficients of a deterministic wavelet series with unbounded I.I.D. random variables.
In particular, we show that, at the opposite to what happens for Fourier series, the randomization of almost every continuous function gives an almost surely nowhere locally bounded function.

%Given a sequence of symmetric and independent random variables, the properties of the corresponding random Fourier series have been deeply investigated. Most of the results obtained in this context show that the randomization of a Fourier series has a regularization effect. In this paper, we consider similar questions in the case of random wavelet series with independent coefficients and we prove that the conclusion widely differs: We deal in particular with the belonging of the random wavelet series to $L^\infty(\TT)$ and $C(\TT)$ and we prove that the randomization of a regular random wavelet series can be nowhere locally bounded. 
 \end{abstract}


 \section{Introduction}
 
 The  study of series of functions with random coefficients has a long and rich history, at the interface of harmonic analysis and functional analysis.  It can be traced back to the seminal, but cryptic,  statement in 1896 by E. Borel that ``in general'', a Taylor series is not continuable  across  its circle of convergence, see e.g. \cite{Borel,Kahane1}.  Though the basic definitions of probability theory hadn't been coined at that time,  this statement was interpreted as stating that the coefficients are random with independent phases,  in which case the conclusion was proved by H. Steinhaus in the 1930 \cite{Steinhaus}.  This topic was later developed by  Paley and Zygmund,  who studied  Rademacher Fourier series, the coefficients of which are of the form $\pm a_n$, where $\pm$ denotes a sequence of independent Rademacher  variables (taking values $+1$ and $-1$ with probability $1/2$).  
 A typical result   they obtained is that  for all $ p \in [1, \infty )$,  the condition $\sum |a_n|^2  < \infty$  is  a necessary and sufficient condition  to ensure that almost surely,  the sample paths of the Rademacher Fourier Series belong to $L^p$ \cite{PZ}.  Many results  on random Fourier series   yield that the randomization of  coefficients has a regularization effect. One of the most striking evidence is supplied by the following theorem, see \cite{Kahane}. 
 
 \BT \label{theo1}  Let  $(X_n)_{n \in \ZZ}$ be a sequence of independent symmetric complex random variables  of variance 1 and let  $(a_n)_{n \in \ZZ}$ be a sequence of complex numbers.  Let
 \BE \label{sj} \forall j \in \N , \qquad s_j = \left( \sum_{ 2^j \leq |n|<2^{j+1} }  | a_n|^2  \right)^{1/2} . \EE 
 If the sequence $(s_j)_{j \in \N}$ is decreasing  and belongs to $\ell^1$, then almost surely the  random trigonometric series 
 \BE \label{rfs} \sum_{n \in \ZZ}  a_n X_n   e^{inx}  \EE
 is continuous. 
 \ET
 
 Note that the condition $\sum_j s_j < + \infty$ by itself is far from implying the continuity of the  associated deterministic Fourier series $\sum a_n e^{inx}$, hence the smoothing consequence of randomization. %Let us already mention that we will consider an explicit example  in Section \ref{seccont} on which we will compare the consequences of the randomisation of Fourier series vs. wavelet series. 
  In the Gaussian case, the exact condition for  a.s. continuity of sample paths  of \eqref{rfs}  was obtained by   Marcus and Pisier;   
% Let  $\chi_n$ be  I.I.D centered Gaussian variables;   if the  dyadic averages 
% \BE  \label{mj} m_j = \left( \sum_{n= 2^j}^{2^{j+1}1} | a_n |^2 \right)^{1/2}  \EE
% form a sequence in $l^1$, then the sample paths of the  process 
% \[ \sum \chi_n a_n e^{inx} \]
% are a.s.  continuous functions 
let us also mention  another famous result which they derived  \cite{MarPis, MarPis2}. 

\BT  \label{proppis} Let $(a_n)_{n \in \ZZ}$ be a sequence of complex numbers and consider the two random Fourier series 
 \BE \label{rfs1} \sum_{n \in \ZZ}  a_n  B_n e^{inx}  \quad \mbox{  and  } \quad   \sum_{n \in \ZZ}  a_n  \chi_n e^{inx} \EE 
where  $(B_n)_{n \in \ZZ} $ is a sequence of independent Rademacher random variables and $(\chi_n)_{n \in \ZZ}$ is a sequence of  independent standard Gaussian random variables. Then, almost surely,  either both of them are continuous or both of them are not bounded. 
\ET

Our purpose in this paper is to investigate what such properties become for  general {\em random wavelet series}, i.e. series on orthonormal wavelet bases with random coefficients  which are independent (note that this terminology has been used previously in the less general setting where all wavelet coefficients at a given scale $j$ are I.I.D., see \cite{AJ02}).  We will  compare them with   results obtained previously for  Fourier series.
Actually, comparing  Fourier vs. wavelet expansions has a long history: it  can be traced  back to 1909 with the introduction of the Haar system, especially devised to supply a basis where the partial sums of the decomposition of a continuous function converges uniformly (in sharp contradistiction with F\'ejer's theorem which states that the Fourier series of a continuous function may diverge at certain points). Later, in  the 1940s,  some properties of    Brownian motion were derived from the initial construction of Wiener on the trigonometric system, whereas  sharp global and pointwise regularity properties were derived from the alternative decomposition of Ciesielski on the {\em Schauder basis} \cite{Cie61}; indeed, this basis  can be viewed as a special (non-orthogonal) wavelet basis, where the generating wavelet is the ``hat'' function 
 \[  \psi (x) =  1_{[0,1] } (x) \cdot  \min (x, 1-x) . \]  Since the introduction of smooth wavelet bases in the mid 1990s,  it has been noted many times that the functional  properties of trigonometric and wavelet expansions widely differ. In short, randomization of trigonometric series has a ``smoothing'' effect whereas it is not the case for wavelet series: This is  a consequence of the fact that both global and pointwise regularity conditions can be characterized by conditions on the moduli of the wavelet coefficients, whereas it is not the case for Fourier series, see \cite{JaffMey1,JaffToul} and ref. therein.

%One of our purposes is to show that the conclusions widely differ for wavelet series:  Proposition \ref{contreex} will supply an example of  wavelet series with Rademacher coefficients which  almost surely has continuous sample paths, whereas the same  series,  with standard Gaussian coefficients instead is nowhere locally bounded. \\ 

 Another classical topic  where   series of functions with coefficients $\pm 1$ show up is the study of expansions on bases in function spaces. 
Let $X$ be a separable  Banach space; a sequence $(e_n)_{n \in \NN} $ of elements of $X$  is a {\em Schauder  basis} 
 if, for any 
 $ f\in X$, there exists a unique  sequence of real numbers $ (a_n)_{n \in \NN} $ such that 
 \BE \label{convstr}
\sum_{n\leq N}  { a_n e_n} \rightarrow f \quad \mbox{ in} \quad  X.\EE 
% \begin{equation} \label{sb}  \left\| \sum_{n= 1}^N  { a_n e_n} -f  \right\|_E \longrightarrow 0 \quad \mbox{ when } N \rightarrow + \infty .  \end{equation} 
The sequence $(e_n)_{n \in \NN} $ forms an  {\em  unconditional  basis} if convergence also  takes place after any permutation of the elements of the series. This is equivalent to the fact that  the series $\sum_{n} \varepsilon_n { a_n e_n}  $ converges in $E$ for any choice of signs $\varepsilon_n= \pm 1$; this is also   equivalent to the existence of a constant  $C >0$ such that for any finite subset $F \subseteq \NN $, any real numbers 
 $ (a_n)_{n \in F} $,  and  any choice of signs $\ep_n = \pm 1 $,  
\begin{equation}  \label{defunco} \left\| \sum_{n \in F} { \ep_n a_n e_n} \right\|_X  \leq C \left\| \sum_{n \in F} { a_n e_n} \right\|_X . \end{equation}
%The best constant $C$ in  \eqref{defunco} is called the {\em unconditional basis constant.} 
Unconditional convergence   insures the numerical stability of the reconstruction of  a function $f$ using linear combinations of the elements $e_n$. 
% A key consequence of \eqref{defunco}  is that, if $f = \sum a_n e_n$, then the norm of $f$  in $E$ is equivalent to a quantity  which  only depends  on the $|a_n|$.  This means that $E$ is isomorphic to a sequence space. 
In  statistics, \eqref{defunco} is often referred to as the {\em multiplier property}. Note that this property is only one of the two ingredients for an unconditional basis and, in particular, it may hold in spaces that are not separable (in which case the first condition cannot be fulfilled); it is for example the case for the H\"older $C^\al$ spaces:  indeed one easily checks that smooth orthonormal wavelet  bases are unconditional weak$^*$ bases (this means that the requirement of strong convergence in \eqref{convstr} is replaced by a   weak$^*$  convergence, see  \cite{Heil}).

Such questions actually motivated the introduction of  wavelet bases: The first ones  were introduced  by J. O. Str\"omberg precisely in order to construct   bases which are simultaneously unconditional for the real Hardy spaces $H^p$  (if $p \leq 1$) and Lebesgue spaces $L^p$ (if $p > 1$), see  \cite{Strom}. 
%Similarly, the   $L^p$ or $C^\al$ norms  cannot be characterized by a quantity bearing on the moduli of the Fourier coefficients (except, of course,  for $L^2$).  
%This shows a first difference between random Fourier and wavelet series: 
In particular, the wavelet characterizations of the $L^p$ spaces for $p>1$ depend on the moduli of wavelet coefficients, and thus are  unaltered by multiplication by Rademacher random variables. Therefore the Paley-Zygmund theorem couldn't possibly hold for  wavelet series. This  shows another important difference between random Fourier and random wavelet series.

Since the H\"older, Sobolev and Besov norms can be characterized by a quantity bearing on the moduli of the wavelet coefficients,  a key consequence of the multiplier property for wavelet methods in statistics is that the shrinkage of wavelet coefficients in the expansion of a function does not introduce spurious singularities in the function, see e.g. \cite{DJKP95} and references therein.
Let us underline  that it  is very specific to wavelets and does not hold for other ``classical'' bases. In particular, in the periodic case, the famous Kahane-Katznelson-DeLeeuw theorem shows that given any function in $L^2$, one can construct a continuous function by increasing the modulus of the Fourier coefficients of $f$, see \cite{KKdL}.

%\BT \label{thm:KKDL}
%Given a positive sequence $(a_n)_{n \in \N}$ of $\ell^2$, there exists a real sequence $(b_n)_{n \in \N}$ such that $|b_n|\geq a_n$ for all $n \in \N$ and such that the  Fourier series
% $$ \sum_{n \in \N} b_n \cos(nx) $$ is continuous. 
% \ET

Note that this result has been extended by F. L. Nazarov to expansions on fairly general orthonormal bases $\psi_n$ which however have to satisfy 
\[ \exists C >0, \quad \forall n, \qquad  \| \psi_n \|_1 \geq C,\]
a condition which is clearly not satisfied by wavelet bases, see \cite{Naza,Maurey}. 

The explicit example that will be worked out in Section \ref{seccont} will illustrate this strong difference between Fourier and wavelet series: We will compare these two   randomizations of the sawtooth function $\{ x \}$ (i.e. the fractional part of $x$). 
\\ 

A weaker condition than unconditional convergence has been introduced  in \cite{BKP85}. Instead of asking that the series converges for any choice of signs $\pm$, one requires only the {\em random unconditional convergence}: The Schauder basis $(e_n)_n \in \N$ is called a \emph{RUC system} in the space $X$ if, for every  $f = \sum_{n \in \N} a_n e_n $,  the convergence  the series also holds for almost every choice of signs $\varepsilon_n = \pm 1$, where  $(\varepsilon_n)_{n \in \N}$ denotes a sequence of independent  Rademacher random variables. The case $L^1$ is particularly interesting: it is well known that $L^1$ has no unconditional basis and moreover, it has no RUC system \cite{BKP85}; therefore, contrary to $L^p$ spaces, though wavelets form a Schauder basis of $L^1$,   this space  is not stable for randomization of wavelet coefficients by I.I.D. Rademacher random variables. 
\\

It is also natural to investigate the case where the Rademacher random variables are replaced by a sequence of I.I.D. random variables (either for Fourier series, or for bases in Banach spaces). Let us focus on the Gaussian case, which has been the subject of the most developed investigations.  As regards Fourier series,  if $X = C (\TT)$, then the result of  Marcus and Pisier (Theorem \ref{proppis}) states that one can reduce the study to the Rademacher case. 
%Moreover, the result of Paley and Zygmund  shows that the a.e. convergence of  Gaussian Fourier series is a  necessary and sufficient condition to insure that this random Fourier series  belongs to $L^p$ for every $p \ge 1$.
The  situation for more general bases in a Banach space is more delicate.
%As regards the result of  Marcus and Pisier stated in Prop. \ref{proppis},  we will see that the conclusion differs  for  random wavelet series. 
%The case of centered Gaussian random variable   $\chi_n$ has been specially investigated: 
%In the case of trigonometric series, a typical  result is that if  $\sum |a_n|^2  < \infty$ then the sample paths of the process  \[ \sum_{n \in \ZZ}  \chi_n a_n e^{inx} \] a.s. belong to  belong to all $L^p$ spaces for $ p < \infty$ \cite{}.  
%The conservation of the  norm   when multiplying the coefficients by I.I.D.  Gaussian random variables has also been shown to hold for many bases in function spaces. 
The next result  gives a general criterion, see \cite{Kva}.

\begin{Theo}
Let $X$ be a Banach space and $(\xi_n)_{n \in \N}$ be a sequence of independent standard Gaussian random variables. The following assertions are equivalent : 
\begin{enumerate}
    \item For all sequence $(x_n)_{n \in \NN}$ of elements of $X$, the series $\sum_{n=0}^{\infty}  \xi_n x_n$ converges almost surely unconditionally if and only if the series $\sum_{n=0}^{\infty} x_n $ converges unconditionally in $X$.
    \item The space $X$  does not contain $\ell_{\infty}^n$ uniformly.
\end{enumerate}
\end{Theo}
Let us recall that $X$ contains $\ell_{\infty}^n = (\mathbb{R}^n, \Vert \cdot \Vert_{\infty})$ uniformly if
$$
\sup_{n \in \mathbb{N}} \delta(X,\ell_{\infty}^n) < + \infty
$$
where $\delta$ is defined for  normed spaces  $E$ and $X$ such that $X$ contains at least one vector subspace isomorphic to $E$ by
$$
\delta(E,X) =\inf \{ d(E,F): \, \, F \subset X, \, F \, \text{isomorphic to } E\}
$$
and $d$ is the Banach-Mazur distance defined between isomorphic spaces $E, F$ as follows
$$
d(E,F)= \inf \{ \Vert T \Vert \Vert T^{-1} \Vert \, : \quad T : X \to Y \text{ isomorphism} \}. 
$$
Note that the spaces $L^p$ for $1\leq p<+\infty$ do not contain $\ell_{\infty}^n$ uniformly. Hence, for $1<p<+ \infty$, the unconditional convergence of the wavelet series implies the almost sure unconditional convergence for the randomized Gaussian expansion. At the opposite $C(\TT)$, $C^\al$  and $L^{\infty}$ contain $\ell_{\infty}^n$ uniformly, and this motivates our investigation of the differences between the convergences properties  of random Fourier and wavelet series in these specific settings.\\
%It motivates the question of  studying the boundedness and continuity of gaussian wavelet series that will be tackle in this paper.  \\
 %\begin{tempo} 
 %L'equivalence donne-t-elle le resultat directement pour la reciproque de la prop 2.11?
 %celine: je ne suis pas sure car rien ne dit que la serie donnee par la negation de l'equivalence sera une serie d ondelettes
 %Question: quel est l'interet du beta dans la definition?
 %\end{tempo}

  Concerning  Gaussian series in $L^1$, a contraction argument allows Ledoux and Talagrand \cite[Section 4.2]{LedouxTalagrand} to show that the average of Gaussian randomization always dominates the corresponding Rademacher average. 
  %More precisely, if $(\chi_n)_{n \in \N}$ is a sequence of independent symmetric random variables such that $  \mathbb{E}(|\chi_n|)< + \infty$ for all $n \in \N$, then for any sequence $(x_n)_{n \in \N}$ is a Banach space $X$, one has
 % \BE \label{eq:ledouxtalagrand}
 %\inf_{n \in \N}\mathbb{E}(|\chi_n|)\,   \mathbb{E} \left( \left\| \sum_{n \in \N} \varepsilon_n x_n\right\|_X \right) \leq 
 % \mathbb{E} \left( \left\| \sum_{n \in \N} \chi_n x_n\right\|_X \right)
%\EE
%where $(\varepsilon_n)_{n \in \N}$ denotes a sequence of independent Rademacher random variables. 
In particular, if a Gaussian series $\sum_n \xi_n x_n$  in a Banach space $X$ converges, so does the corresponding Rademacher series $\sum_n \varepsilon_n x_n$. Together with the result of   \cite{BKP85} about RUC systems, it implies that the space $L^1$ is not closed under Gaussian randomization of the wavelet coefficients.

Non-Gaussian randomization has been  the subject of much  less  investigations. Let us however mention the following result of \cite{Kva} which gives a criterium for unconditional convergence.

\begin{Theo}
    Let $X$ be a Banach space that does not contain $\ell_{\infty}^n$ uniformly, $(\zeta_k)$ be a sequence of numerical variables for which 
    $$
    \sup_{k \in \mathbb{N}} \mathbb{E} \vert \zeta_k \vert^p \le c_p < +\infty \quad \forall p >0
    $$
    and $a_k \in X, k= 1,2,...$. If the series $\sum_k a_k$ converges unconditionnaly, then the series $\sum_k a_k \zeta_k$ converges almost surely unconditionnaly in the space $X$.  
\end{Theo}

 With this context in mind, our paper will deal with the study of the boundedness,  continuity and H\"older regularity of random wavelet series. In Section \ref{sec2}, we give a  natural condition for the continuity of a function in terms of its wavelet coefficients. We moreover prove that this condition is optimal. Section \ref{sec3} is dedicated to the study of nowhere bounded random wavelet series. In particular, we prove that the conclusion of  Theorem \ref{proppis} widely differ for wavelet series:  Proposition \ref{contreex} will supply  examples of  wavelet series with Rademacher coefficients which  almost surely has continuous sample paths, whereas the same  series,  with standard Gaussian coefficients instead is nowhere locally bounded (we will actually prove that Rademacher and Gaussian random variables play no specific role in the wavelet case, but that the dichotomy actually is between bounded vs. unbounded random variables). Moreover, we prove that, far from being exceptional, this behavior is ``generic'' in the space of continuous function. 
 We  consider   in Section \ref{sec5} a minimal regularity condition implying the continuity of the randomized wavelet series. We finally  study in Section \ref{seccont}   the impact of a randomization of wavelet coefficients on the H\"older modulus of continuity of the function and   
we propose an explicit example on which we will compare the consequences of the randomization of Fourier series vs. wavelet series. 

 \section{Continuity of wavelet series}\label{sec2}

 In this section, we give a simple condition in terms of the  wavelet coefficients of a function that implies the continuity of this function. Furthermore, we show that this condition is  optimal. Let us start by recalling the definition of a $d$-variables wavelet basis \cite{Dau92,Mey90I,Mallat1998}.

 \BD  \label{defwavsyst}  An orthonormal $r$-smooth ($r \geq 0$) wavelet basis on $\RR^d$ is a collection of   functions  $\varphi$ and $\Psi^{(i)}$,    $ i \in \{ 1, \dots, 2^d -1 \} $,  with the following properties:\begin{itemize}
\item $\varphi$ and the $\Psi^{(i)}$ and all their  partial derivatives  up to the order $r$   have fast decay, 
\item The $\varphi (\cdot -k )$,   $k \in \ZZ^d$  together with the functions 
 $  2^{dj/2} \Psi^{(i)}  (2^j \cdot -k)$, $j \in \NN, \; k \in \ZZ^d$, $i  \in \{ 1, \dots, 2^d -1\}$,  
 form an orthonormal basis of $\RR^d$. 
\end{itemize}
  \ED  

  For example, the Haar system is a $0$-smooth wavelet basis. We will usually not mention the smoothness required for the wavelet basis, assuming that it is ``large enough''; the precise value required can always be easily tracked.   We will mostly work on the $d$-dimensional torus $ \TT = \RR^d / \ZZ^d$, in order to draw tighter comparisons with Fourier series. Note however that the results we obtain extend easily to the non-periodic setting (in which case they have to be understood as local results).  In that case  {\em periodized wavelet bases} are used; they  are supplied by the union of the constant function $1_{\TT} $, and the 1-periodic wavelets $\ptjk$ defined on $\RR$ by 
  \[  \ptjk (x) : = \sum_{l \in \ZZ^d}  \Psi^{(i)}  (2^j (x-l) -k) ,    \quad  i  \in \{ 1, \dots, 2^d -1\},  \;  j \in \NN, \; k  \cdot 2^{-j} \in  [0,1)^d . \] 
   We will use the following notations   for wavelets: 
 \[  \varphi_k = \varphi (\cdot - k), \quad  \pjk =  \Psi^{(i)}  (2^j \cdot -k), \quad \ptjk =  {\psi}^{(i)}  (2^j \cdot -k)  , \]
  and by $c_k$ and $c^{i}_{j,k}  $ the corresponding  wavelet  coefficients (the context will tell without ambiguity if the wavelet coefficients are computed in the periodic or non-periodic setting).  Note that we do note use the $L^2$ normalization for wavelets and wavelet coefficients; indeed the convention we use will lead to simpler formulations when dealing with boundedness, continuity, or H\"older regularity questions.  
  
  Let us first recall two straightforward  results.
  
  \BP \label{propcritqsimple}  Let $f \in L^\infty (\TT) $ and  denote the wavelet coefficients of $f$ in a fixed $0$-smooth wavelet basis by $c^{i}_{j,k}$. Then, the sequence  
  $(\ome_j)_{j \in \NN}$ defined by
  \BE \label{omej} \ome_j = \sup_{i,k} | c^{i}_{j,k} |  \EE 
  belongs to $\ell^{\infty}$.
  %\[ \exists C: \forall j \qquad \ome_j \leq C .\]
   Additionally, if $f$ is continuous then $(\ome_j)_{j \in \NN}$ belongs to $c_0$.
  %\[ \ome_j \rightarrow 0  \mbox{ when } j \rightarrow + \infty. \]
  \EP
  The optimality  of the first statement is shown (in the non-periodic case) by the (slight variant of the) Heaviside function 
 \BE \label{heavi1}
 H (x) = 1_{\RR^+} (x) - 1_{\RR^-} (x) , 
 \EE
   which clearly   satisfies 
   \BE \label{heavi2}  c_{j,k} = c_{j+1,k} . \EE 
   Recall that the {\em  cone of influence} of width $C$ at $x_0$  is defined as 
   \[ { \mathcal C} (x_0, C) = \{ (j,k): \quad | 2^j x_0 -k| \leq C\} . \] 
   Hence,  \eqref{heavi2} shows that   the  wavelet coefficients in the  cones of influence at 0  do not decay (except in the case of the Haar wavelet, where they all vanish!). Adapting this counter-example to the periodic and $d$-dimensional settings is straightforward.

%\subsection{ Continuity of wavelet expansions} 
%\label{seccont}

% The purpose of Section \ref{seccont} will be  to discuss the optimality of  a  wavelet continuity criterium, supplied by Prop. \ref{prop1}  below. 
Since the space ${ \cal C}  (\TT^d) $ (composed of continuous bounded functions on the $d$-dimensional torus)  has no unconditional basis, and since wavelets form a Schauder basis of this space, it follows that the multiplier property does not hold here.

 Nonetheless the following result supplies a simple and optimal condition implying continuity. 

\BP \label{prop1}  Let $c^{i}_{j,k}$ be a collection of coefficients with  $i  \in \{ 1, \dots, 2^d -1\},$    $  j \geq 0, $ and $ k \cdot 2^{-j} \in  [0,1)^d $   and assume that the wavelets used are 0-smooth. 
\begin{itemize}
\item If the  sequence $(\ome_j)_{j \in \N}$ defined by \eqref{omej} belongs to $\ell^1$,  then  the wavelet series 
\begin{equation} \label{wavser} f=  \sum_j f_j \quad \mbox{ where} \quad f_j = \sum_{i,k}   c^{i}_{j,k} \psi^{(i)}_{j,k}  ,   \end{equation} 
is normally convergent, so that it  defines a bounded function. If  additionally   the wavelets are continuous, then $f$ 
is  also continuous.
\item 
This condition is optimal, i.e. if   $(\ome_j)_{j \in \N}$ is a non-negative sequence such that $\sum_{j \in \N} \ome_j  =  \infty $  and if the wavelets are piecewise continuous, then there exists a wavelet series \eqref{wavser}  such that 
\BE \label{eq:om}
\forall j, \quad \sup_{i,k} | c^{i}_{j,k} |  = \ome_j  , \EE
 and  which is  is nowhere locally bounded on $\TT^d $.
 \end{itemize}
\EP

\begin{rem} %In the following, we shall refer to wavelet series \eqref{wavser} satisfying   $(\ome_j)_{j \in \N} \in \ell^1   $ as {\em normally convergent wavelet series}.  
Recall that a  Sidon set of integers is a set $E\subset \ZZ$ such that, if
\[ \sum_{n \in E} c_n e^{inx} \]
is continuous, then $\sum_{n \in E} | c_n | < \infty$, see \cite{Rud}.  The determination of Sidon sets  has a long and rich history in harmonic analysis. One can see the quest for Sidon sets as understanding which lacunarity condition on a Fourier series implies that the ``natural condition '' of absolute convergence turns out to be necessary. Therefore the converse part of our proposition can be interpreted as a ``wavelet variant'' of this quest. More precisely, the natural counterpart of the Sidon sets question in the wavelet setting would be to determine sets of indices ${ \mathcal I}$  such that 
\[ \sum_{(i,j,k)  \in { \mathcal I} }  c^{i}_{j,k}  \psi^{(i)}_{j,k}  \in L^\infty \Longrightarrow  \sum_{(i,j,k)  \in { \mathcal I} }  |c^{i}_{j,k}| | \psi^{(i)}_{j,k} | \in L^\infty . \] 
\end{rem}

%\begin{tempo} 
%Le cas Haar pourrait avoir ete regarde: le point de vue ``martingales dyadiques'' a ete tres explore   \end{tempo} 

\begin{proof}[Proof of the first point] The direct sense of the implication is straightforward: Indeed, 
the fast decay of the wavelets implies that, if the wavelets are continuous, then the $f_j$ are continuous. In all cases,  one has
\[ \left| f_j (x) \right|  \leq C \ome_j \quad \forall j \in \NN .   \] 
It follows that the series  of functions $\sum_j f_j $ is a series of bounded  functions which converges normally, so that its sum is bounded. Additionally,  if the wavelets are continuous, then the $f_j$  are continuous (because of the fast decay of the wavelets)  and the sum of the series is  a continuous function.  
\end{proof}

In order to prove the converse result, we note that %we can assume that $\sum_j \ome_j  =  \infty $. We 
we can assume that 
\BE \exists C \; \forall i, j, k , \quad | c^{i}_{j,k} | \leq C  \EE
since Proposition \ref{propcritqsimple}  states  that if  $f \in L^\infty$, then this condition is satisfied.  
We will use the following lemma.

\BL \label{Elem1} If   $\ome_j$ is a nonnegative sequence such that $\sum_{j \in \NN} \ome_j  =  \infty $, then  there exists  a subsequence $(j_n)_{n \in \N}$ such that 
\BE \label{lem1}  j_{n+1} -j_n  \rightarrow + \infty \quad \mbox{ and } \quad  \sum_{n \in \N} \ome_{j_n}   =  \infty. \EE 
\EL 

\begin{proof} The result is immediate if  we fix an $N$ and  we only require that $ j_{n+1} -j_n = N $ (indeed, one of the $N$ sequences which  satisfy this property necessarily satisfies $ \sum_n \ome_{j_n}   =  \infty$).  In order to obtain the subsequence such that \eqref{lem1} holds, one starts by taking element of the sequence  $j_n$  thus obtained for $N =2$ until the corresponding sum is larger that 1, then one takes elements for $N = 3$  until the corresponding sum is larger that 1, and so on.  \\ 
\end{proof}

We now come back to the proof of the optimality  of the proposition.

\begin{proof}[Proof of the second point]
First we construct a wavelet series satisfying \eqref{eq:om} which does not belong to  $L^{\infty}(\TT^d)$: The result is clear if the $\ome_j$ are unbounded, therefore we can assume that the sequence $\ome_j$ belongs to $l^\infty$.  At each scale $j$, we consider only one nonvanishing coefficient defined by   
\[  c^{1}_{j,k}  = \ome_j.  \] 
The localization of these coefficients is chosen as follows. First, if $j $ is not one of the $j_n$, we locate these coefficients at a $k= (k_1, \cdots k_d) $ such that $\forall m $, $ k_m \cdot 2^{-j}  = 3/4$. 
Secondly, for the localization at scales $j_n$, $n \in \N$,  we notice that the assumption of piecewise continuity of the wavelets implies  that there exists $C >0$ and a dyadic cube $K$   such that 
\[  \psi^{(1)} (x) \geq C , \quad \forall x \in K.  \] 
We denote by $K_{j,k}$ the dyadic cube 
\[ K_{j,k} = \{ x : \, 2^j x-k \in K \}. \] 
If  $Q$ is a dyadic cube, we denote by $Q/2$ the cube with same center, same orientation, and width $| Q/2 |= 1/2 |Q|$. 
With these notations, if  $j $ is  one of the $j_n$, the localization of the nonvanishing coefficients is given by the first couple $(j_n,k_n)$  so that the corresponding cube $K_{j_n,k_n} $ is included in $[1/8, 3/8 ]^d$, and the following ones so that  
\BE \label{inclus} K_{j_{n+1},k_{n+1} } \subset 1/2 \cdot K_{j_n,k_n} \EE  (the condition $j_{n+1} -j_n  \rightarrow + \infty$ implies that this is possible).  

We now consider the wavelet series \[ f=  \sum_{j}  c^{1}_{j,k} \pjkun     \]  thus constructed. 
Since the $c^{1}_{j,k}$ are bounded, and only one does not vanish at each scale, this function belongs to $L^2$. Furthermore, because of the fast decay of the wavelets, the contribution of the $j$  that do not belong to the sequence $j_n$ is bounded on the cube $[1/8, 3/8 ]^d$, and we can forget it.  We now focus on the values of the remaining term 
\BE\label{gpart} g =  \sum_{n}  c^{1}_{j_n,k_n} \psi^{(1)}_{j_n,k_n}    \EE
on the  cubes $K_{j_l,k_l} $. If $x \in K_{j_l,k_l}$, we have
\[    c^{1}_{j_n,k_n} \psi^{(1)}_{j_n,k_n} (x) \geq  C \ome_{j_n} \]
for every $n \leq l$. 
First assume that the wavelets are compactly supported. For $n > l$, it follows from \eqref{inclus} and the condition $j_{n+1} -j_n  \rightarrow + \infty$ that the support of the wavelet $\psi^{(1)}_{j_n,k_n} $ is included in $K_{j_l,k_l} $. Since wavelets have a vanishing integral, it follows that the average  value of $g$   on  $K_{j_l,k_l} $ is larger than 
\BE \label{valmoy} C \sum_{ n=1}^l   \ome_{j_n}.\EE
Since these quantities are not bounded, it follows that $g \notin L^\infty$. 

Let us now only assume that the wavelets have fast decay. Then the average value of  the wavelet $\psi^{(1)}_{j_n,k_n} $  on  $K_{j_l,k_l} $ with $n >l$ does not exactly vanish. However, using again that wavelets have a vanishing integral, it is bounded by  
\[ \int_{ x \notin K_{j_l,k_l}} |\psi^{(1)}_{j_n,k_n} (x) |dx,  \]
and  the conditions $j_{n+1} -j_n  \rightarrow + \infty$  and \eqref{inclus} imply that these quantities are bounded $C /(n +l)^2$;  therefore \eqref{valmoy} is only modified by a bounded  error term, and the same conclusion holds. 

Hence, we have obtained that  the function $f$ can be written as $f = g+h$ where 
$$
h= \sum_{j \neq j_n} c^1_{j,k} \psi^{(1)}_{j,k}
$$
and $g$ given by \eqref{gpart}
are bounded everywhere on $\TT^d$ except at at most one point for $h$ and one point for $g$. We define 
$$
g_l = \sum_{n \ge l+1} c^1_{j_n,k_n}\psi^{(1)}_{j_n,k_n}.
$$
Then
$$
G(x) = \sum_{l \ge 0} \sum_{k \notin (2 \mathbb{Z})^d} g_l(x-2^{-l}k+2^{j_l}k_l)
$$
is nowhere locally bounded and satisfy $\omega_j(G)= \omega_j(g)$, where $\omega_j(\phi)$ is defined by $\omega_j(\phi) = \sup_{i,k} \vert c^i_{j,k}\vert $ for a wavelet serie $\phi= \sum_{i,j,k} c^i_{j,k} \psi^{(i)}_{j,k}$. Finally $F = G+h$ remains nowhere locally bounded and satisfy $\omega_j(F)= \omega_j$.
 \end{proof}
 
\section{Nowhere locally  bounded  random wavelet series } \label{sec3}

As already mentioned, an important feature of  randomization of Fourier coefficients is a smoothing effect. This section aims at showing that the 
% The Marcus and Pisier theorem states that under the weak condition that the $m_j$ defined by \eqref{mj} form an $l^1$ sequence, then the associated trigonometric series, which, in general, does not represent a bounded function, becomes a.s. continuous after a Gaussian randomization of its wavelet coefficients. 
  conclusion can go in the opposite direction for wavelet series: We will construct a bounded function such that, after an unbounded randomization of its wavelet coefficients, the sample paths of the stochastic process thus obtained  a.s. are nowhere locally bounded. Let us be more precise; 
we consider general {\em  random wavelet series} i.e. stochastic processes of the form 
\BE \label{RWS}  X =  \sum_{i,j,k}  c^{i}_{j,k} \chi^{i}_{j,k} \ptjk ,  \EE 
where the   $\chi^{i}_{j,k}$ are I.I.D.  random variables of common law $\chi$.  Our purpose is to compare the continuity and boundedness properties of the series 
\BE \label{NorConv}   f =  \sum_{i,j,k}  c^{i}_{j,k}  \ptjk   \EE 
and of its randomization \eqref{RWS}. We assume in this section that the series \eqref{NorConv} is {\em normally convergent} considered as a series in $j$ only (with a slight abuse of language), i.e.  that 
\[ \sum_j \sup_x \left( \sum_{i,k}  | c^{i}_{j,k}|  | \ptjk   (x) | \right)  < \infty  .  \]
Such a function clearly is continuous  if the wavelets are continuous (which we assume in this section) and,  if $\chi$ is bounded,  then its randomization \eqref{RWS} also  is normally convergent and therefore it represents a stochastic process with continuous sample paths. The purpose of this section is to prove that this result may fail if the random variable $\chi$ is  not bounded.

  We start by a simple remark  on unbounded random variables.
  First, note that 
  \[ \forall n, \qquad \PP ( | \chi | \geq n^3 ) >0, \]
  so that there exists an increasing sequence $(j_n)_{n \in \N}$ such that 
  \BE \label{proban3} \PP ( | \chi | \geq n^3 ) \geq 2^{-dj_n}.\EE
  Therefore, since there are $(2^d-1) 2^{dj_n}$ wavelets at scale $j_n$,  
  \[ \sum_n \sum_{i,k} \PP ( | \chi^{i}_{j_n, k} | \geq n^3 )  = + \infty , \]
and the Borel-Cantelli Lemma yields that a.s. there exists an infinite number of  $| \chi^{i}_{j_n, k} | $  larger than $n^3$.  We refer to $(j_n)_{n \in \N}$ as a {\em divergence sequence} associated with the law $\chi$.

We can now state the main result of this section,  which gives  an  explicit example of a  continuous function such that its randomized wavelet coefficients yield a nowhere locally bounded stochastic process if the random variable $\chi$ is unbounded.

 \BP \label{contreex}
  Let $\chi$ be an unbounded random variable, and let $(j_n)_{n \in \N} $ be a divergence sequence associated with $\chi$. Then  
\[
f = \sum_{n\in \N} \frac{1}{n^{2}}   \sum_{i, k}  {\psi}^{(i)}_{j_n ,k} 
\]
is a normally convergent wavelet series (and therefore  $f$ is continuous). However  its 
  randomization
\[
X_f = \sum_{n\in \N} \frac{1}{n^{2}}  \sum_{i, k}  \chi^{i}_{j_n, k}{\psi}^{(i)}_{j_n ,k}
\]
is a.s. nowhere locally bounded, where the   $\chi^{i}_{j,k}$ are I.I.D.  random variables of common law $\chi$.
\EP

\begin{rem} The result is stated in the setting of periodized wavelets, but it translates immediately in the setting of standard wavelets on $\RR^d$. 
\end{rem}

\begin{rem} Without any assumption on $\chi$ the series $X_f$ is a formal series. However, if we assume that $\chi$ has a finite variance, then $\EEE (\| X_f\|^2)  < \infty $, so that  the sample paths of $X_f$ belong to $L^2$ almost surely. The proof of continuity also yields that, if the $\chi^{i}_{j_n, k}$ are bounded, then the randomized wavelet series   has continuous sample path. A direct consequence of this theorem therefore is that the Marcus and Pisier theorem does not hold for random wavelet series and, in that case, its failure is related only with the boundedness of the random variable $\chi$ and not its particular nature (Rademacher, Gaussian, ...).  \end{rem}

\begin{proof}
The first statement follows from Proposition \ref{prop1}. 
As regards the second statement, as a consequence of Lemma \ref{Elem1}, it follows that a.s. for $j$ large enough, one of the $| \chi^{i}_{j_n,k}|$  
is larger than $n^3$, so that the wavelet coefficients of $X_f$ are not bounded. Using Proposition  \ref{propcritqsimple}, it implies that $X_f$  does no belong to $L^\infty$. In order to obtain the more precise result that $X_f$ is not locally bounded, we note that the argument which gave the divergence sequence  can be strengthened as follows: Let us pick a dyadic cube $\la \subset [0,1]^d$; the same argument yields that, if we consider the  dyadic cubes included in $\la$, then a.s. there exists  an infinite number of   $\chi^{i}_{j_n,k_n}$ corresponding to  indices $(j_n,k_n)$ such that $k 2^{-j_n}  \subset \la$  and which satisfy $| \chi^{i}_{j_n, k} | \geq n^3$.  This is true for all dyadic subcubes of  $[0,1]^d$, which form a countable set. It follows that $X_f$ actually is nowhere locally bounded.
\end{proof}

Let us now prove that the behavior of this example is not
pathological, but that, on the opposite, it is generic  in the space $C(\mathbb{T}^d)$  using the notion of prevalence.
The notion of prevalence supplies an extension of the notion of ``almost everywhere'' (for the  Lebesgue measure)  in an infinite dimensional setting. In
a metric infinite dimensional vector space,  no measure  is both  $\sigma$-finite  and translation invariant. However, a natural notion of ``almost everywhere'' which is translation invariant  can be defined as follows,  see  \cite{Christ,HSY}.

\BD\label{def:prev}
Let E be a complete metric vector space. A Borel set $A\subset E$ is  {Haar-null} if there exists a compactly supported probability measure $\mu$  
such that
$$
\forall x\in E,\quad \mu(x+A)=0. 
$$
  If this property holds, the measure $\mu$ is said to be {transverse}
  to $A$. A subset of $E$ is called {Haar-null} if it is contained in a Haar-null Borel set. The complement of a Haar-null set is called a {prevalent set}.
\ED

%With a slight abuse of language, when a property holds on a prevalent set, we will say that it holds almost everywhere. 
If $E$ is a function space, one can often use for transverse measure $\mu$
the law of a stochastic process, see e.g. \cite{Fraysse,EsserJaffard,EsserLoosveldt} for some
applications of this method.
In this context, a way
 to check that  a property
$\mathcal{P}$ holds  only on a Haar-null set  is to  exhibit  a random process $X$ whose  sample paths lies in a compact subset of $E$
and  such that for all $f \in E$, almost surely the property
$\mathcal{P}$ does not hold for $X+f$. 

%Now, if $f$ is a continuous function, we denote by $(c_{j,k})_{j,k}$
%its sequence of wavelet coefficients and by $f^{\rm gauss}$ the formal
%series given by
%$$
%f^{\rm gauss} := \sum_{j \in \N} \sum_{k=0}^{2^{j}-1} \chi_{j,k} c_{j,k} \psi_{j,k}
%$$
%where $(\chi_{j,k})_{j,k}$ is a sequence of independant
%$\mathcal{N}(0,1)$ random variables on a probability space $(\Omega,
%\mathcal{F}, \mathbb{P})$.

\begin{Theo}
Let $(\chi^{i}_{j,k})_{j,k}$ be a  sequence of independent
I.I.D. unbounded random variables of common law $\chi$. 
For almost every function $f$ in $C(\TT^d)$, the associated randomized wavelet series is almost surely nowhere locally bounded. 
\end{Theo}

\begin{proof}
We assume  that $d=1$ (the proof for $d >1$ is similar). 
Let $(\Omega,
\mathcal{F}, \mathbb{P})$ denote the probability space on which the random variables $(\chi_{j,k})_{j,k}$  are defined. Let us consider an increasing sequence $(j_n)_n$ such that 
  \BE \label{proban3bis} \PP ( | \chi | \geq n^3 ) \geq j_n2^{- j_n}.\EE
Let $(\varepsilon_{j_n,k})_{j_n,k}$ be a sequence of independent Rademacher random variables on a second probability
space $(\Omega',\mathcal{F}', \mathbb{P}')$ and  let $X$ be the stochastic
process defined on   $(\Omega',\mathcal{F}', \mathbb{P}')$ by
$$
X = \sum_{n \in \N}\sum_{k=0}^{2^j-1}
\frac{\varepsilon_{j_n,k}}{n^{2}} \psi_{2^{j_n},k}.
$$
As $|\varepsilon_{n}|=1$, we get as previously that the sample paths
of $X$ define a pointwise bounded subset of $C(\TT)$ and moreover, that  this subset is
equicontinuous.  The Arzela-Ascoli theorem implies that the sample paths
of $X$ are included in a compact subset of $C(\TT)$.

Let $f$ be an arbitrary function of $C(\TT)$ and 
denote by $(c_{j,k})_{j,k}$ its sequence of wavelet coefficients. Let us prove that $\mathbb{P}'$-almost surely, the random wavelet series
$$
\sum_{n \in \N}\sum_{k=0}^{2^{j_n}-1}
(\frac{\varepsilon_{j_n,k}}{n^{2}} +c_{j_n,k}) \chi_{j_n,k}\psi_{2^{j_n},k}  + \sum_{j \neq j_n}\sum_{k=0}^{2^j-1} c_{j,k} \chi_{j,k}\psi_{j,k}
$$
associated to 
$X+f$ is $\mathbb{P}$-almost surely nowhere locally
bounded by showing that  this random wavelet series is $\mathbb{P}$-almost surely not bounded on any dyadic interval $\lambda \subseteq [0,1]$. As in the proof of Proposition \ref{contreex}, it suffices to work on $[0,1]$.  
For every $n$, we consider the disjoint subsets of $\{0, \dots, 2^{j_n}-1 \}$ defined by
$$\Lambda_{n,l} = \{l2j_n, \dots, (l+1)2j_n -1 \}$$
for $l \in \{0, \dots, 2^{j_n}/2j_n\}$. Notice that for every $n$, the inequality
$$
\left|\frac{\varepsilon_{j_n,k}}{n^{2}} +c_{j_n,k}\right| < \frac{1}{n^2}
$$
can only happen if $\varepsilon_{j_n,k}$ is the nearest integer  to $n^2c_{j_n,k}$. Since $\varepsilon_{j_n,k}$ can only takes the values $1$ and $-1$, the probability of this event is bounded by $1/2$. Hence, we have
$$
\sum_{n \in \N} \sum_{l=0}^{2^{j_n}/2j_n} \widetilde{\mathbb{P}}\left( \forall k \in \Lambda_{n,l}\, , \, \left|\frac{\varepsilon_{j_n,k}}{n^{2}} +c_{j_n,k}\right| < \frac{1}{n^2}\right) \leq \sum_{n \in \N}\frac{2^{j_n}}{2j_n}2^{- 2j_n}< + \infty
$$
and the Borel-Cantelli lemma implies that $\widetilde{\mathbb{P}}$-almost surely, for all $n$ large enough and all $l \in \{0, \dots, 2^{j_n}/2j_n\}$, there exists $k_{n,l}\in \Lambda_{n,l}$ such that 
$$
\left|\frac{\varepsilon_{j_n,k_{n,l}}}{n^{2}} +c_{j_n,k_{n,l}}\right| \geq \frac{1}{n^2}.
$$
It follows from \eqref{proban3bis} that
$$
\sum_n \sum_l \PP(|\chi_{j_n,k_{n,l}}| \ge n^3) = \infty.
$$
Applying  the Borel-Cantelli Lemma on the space $(\Omega, \mathcal{F}, \mathbb{P})$, we get that a.s. there are infinitely many  $\chi_{j_n,k_{n,l}}$ larger than $n^3$, hence infinitely many values of $n$ such that 
$$
\left|\frac{\varepsilon_{j_n,k_{n,l}}}{n^{2}} +c_{j_n,k_{n,l}}\right| |\chi_{j_n,k_{n,l}}| \geq n. 
$$
\end{proof}

\section{Conditions implying the continuity of Gaussian wavelet randomization} \label{sec5}

We have shown in Proposition \ref{prop1} that if function $f$ satisfies that the sequence $(\omega_j)_{j \in \N}$ defined by \eqref{omej} belongs to $ \ell^1$, then  $f$ is continuous. We have moreover obtained that one cannot expect a less restrictive condition. In Proposition \ref{contreex}, we have shown that the condition $(\omega_j)_{j \in \N} \in \ell^1$ is not sufficient to insure that the randomized version of $f$ belongs to $L^\infty$. The aim of this section is to  determine a minimal regularity condition implying the continuity of Gaussian wavelet randomization, and to study its optimality. 

First, let us recall a classical result about the asymptotic behavior of a sequence of independent standard Gaussian random variable, which is direct consequence of the Borel-Cantelli Lemma  together with the inequality
\[  \PP ( | \chi | \geq x ) \leq \frac{1}{x \sqrt{2\pi}} e^{-\frac{x^2}{2} }  ,  \quad \forall x >0\] 
if $\chi$ is a standard Gaussian random variable. 

\BL \label{lem2} Let  $(\chi^{i}_{j,k})_{i,j,k}$ be a sequence of independent standard Gaussian random variables. Almost surely, there exists $J\in \N$ such that  $$  \sup_{i,k} | \chi^i_{j,k} | \leq \sqrt{2d j} \, , \quad \forall j \geq J. $$
\EL

%We have shown that the Gaussian wavelet randomization of a continuous function may be nowhere locally bounded and, in the opposite direction, if $F \in C^\ep$ for an $\ep >0$, then its Gaussian wavelet randomization  belongs to $C^{\ep'}$ for any $\ep' < \ep$, and, in particular is continuous.  Therefore a natural question is to determine a minimal regularity condition implying the continuity of Gaussian wavelet randomization.

The following lemma is a direct consequence of standard arguments  on random series, see e.g. \cite{Kahane}; we give a short proof for the sake of completeness.

\BL \label{lem:series}
Let $(\chi_j)_{j \in \N}$ be a sequence of independent  standard Gaussian random variables. %such that $0<\inf_{n \in \N}\mathbb{E}(|\chi_n|)< + \infty$.  
If $\sum_{j \in \N}\ome _j  = + \infty$, then   $$\sum_{j \in \N} | \chi_j | \ome _j  = + \infty$$ almost surely. 
\EL

\begin{proof}
First, let us prove the result for I.I.D. Bernoulli  random variables $B_j$ of parameter $1/2$, i.e. such that  such that 
\[ \PP (B_j = 0 ) = \PP (B_j = 1 )= 1/2. \]
Since the convergence of the series $\sum_{j \in \N} B_j \ome_j$  does not depend on the values of the first $B_j$'s. Therefore the Kolmogorov 0-1 law applies, and the probability  that the series converges is either 0 or 1. Assume that it is 0, i.e. that the series converges almost surely. Then it is also the case for the series  $\sum Q_j \ome_j$, where  $Q_j = 1-B_j$ also  are I.I.D. Bernoulli  random variables. But the sum of these two series is the deterministic series $\sum_{j \in \N}\ome _j $ which diverges by assumption. 
Now, let $a>0$ be such that $\mathbb{P}(|\chi|\geq  a ) = 1/2$. Then, $(1_{\{ |\chi_{j}|\geq a\}})_{j \in \N}$ forms a sequence of I.I.D. Bernoulli random variables. Hence, the series
$
\sum_{j \in \N} a 1_{\{ |\chi_{j}|\geq a\}} \omega_j
$
diverges almost surely. The conclusion follows from the relation $a 1_{\{ |\chi_{j}|\geq a\}} \leq |\chi_j|$.
\end{proof}

\BP \label{propantimarpis} Assume that the wavelets  are 0-smooth and continuous. \begin{itemize}
    \item If $f$ is a periodic function and if the quantities $\ome_j$ defined by \eqref{omej} satisfy
\BE \label{racj}
\sum_{ j \in \N} \sqrt{j} \ome_j  < \infty ,  
\EE
   then the Gaussian wavelet randomization $X_f$ of $f$ is continous.
    \item  Conversely, if   $(\ome_j)_{j \in \N}$ is a non-negative sequence such that $\sum_{j \in \N} \ome_j  =  \infty $  and if the wavelets used are piecewise continuous, then  there exists  a function $ f=  \displaystyle\sum_{i,j,k}  c^{i}_{j,k} \ptjk  $ with 
\[
\sup_{i,k} | c^{i}_{j,k} |  = \ome_j  , \quad \forall j \in \N \] 
and such that the wavelet Gaussian randomization $X_f$ of $f$ is a.s. not bounded on $\TT^d$.
\end{itemize} 
  
   \EP

\begin{proof}
The first part of this proposition  is a direct consequence of Proposition \ref{prop1}  and Lemma \ref{lem2}.  The proof of the second part is an adaptation of the proof of  the second part of Proposition  \ref{prop1}, with two modifications: First, we use Lemma \ref{lem:series}. Then, the second modification is that we have to adapt to the fact that the Gaussian random variable  may take positive or negative values. For that, we also use another cube where $K'$ where \[ \forall x \in K', \qquad \psi^1 (x) \leq -C,  \] 
and  denote by $K'_{j,k}$ the dyadic cube 
\[ K'_{j,k} = \{ x : \, 2^j x-k \in K' \}. \]  One then modifies  the proof of  the second part of Proposition \ref{prop1} for the construction of the first function $f$,  by putting one $\omega_{j_{n+1}} \psi^{(1)}_{j_{n+1},k_{n+1}}$  such that $K_{j_{n+1},k_{n+1}} \subset 1/2. K_{j_n,k_n}$ and one $\omega_{j_{n+1}} \psi^{(1)}_{j_{n+1},k'_{n+1}}$  such that $K_{j_{n+1},k'_{n+1}} \subset 1/2. K'_{j_n,k_n}$. 
With this modification, depending if the corresponding Gaussian random variable  is positive or negative, we focus on the part of the function such that $ c^{1}_{j_n,k_n} \chi^i_{j_n,k_n}  \psi^{(1)}_{j_n,k_n}$ is positive and obtain an unbounded series. By the same argument as in the proof of Proposition \ref{prop1}, the other components of $f$ have a smooth contribution on $K_{j_n,k_n}$ and therefore do not modify the unboundedness of the series.
\end{proof}

  \BR \label{rem:chaos1}
  The arguments developed in the previous result can easily be extended to more general laws as follows. 
  \begin{itemize}
      \item Let us consider a law with exponential decay, i.e. a random variable $\chi$ such that
  \BE \label{eq:expdecay}
  \PP(|\chi|\geq x) \leq C e^{-b x^\gamma} , \quad \forall x >0
  \EE
  for some constants $C,b>0$ and $\gamma\in (0,2]$. Then the Borel-Cantelli Lemma yields that  almost surely 
  $$
  \sup_{i,k} | \chi^i_{j,k} | \leq \left(\frac{dj}{b}\right)^{\frac{1}{\gamma}} 
$$
 for every $j$ large enough. Hence, condition \eqref{racj} is replaced by $\sum_{ j \in \N}  j^{\frac{1}{\gamma}} \ome_j  < \infty$. 
      \item  An  assumption  less restrictive than exponential decay is the assumption that the random variable $\chi$ has finite moments of all orders. This is equivalent to the fact that the law has fast decay, i.e. 
  \BE \label{eq:fastdecay}
  \forall l >0 \quad \exists C>0 \quad \text{such that} \quad \PP(|\chi|\geq x) \leq \frac{C}{x^l} , \quad \forall x >0 .
  \EE
  Then the Borel-Cantelli Lemma gives  that  almost surely, for every $\varepsilon>0$, one has
  $$
  \sup_{i,k} | \chi^i_{j,k} | \leq 2^{\varepsilon j}
$$
 for every $j$ large enough. Therefore, condition \eqref{racj} is replaced by the existence of $\varepsilon >0$ such that 
 $\sum_{ j \in \N}  2^{\varepsilon j} \ome_j  < \infty$. 
  \end{itemize}
  In both cases, for the second part, if there is $a>0$  such that $\mathbb{P}(|\chi|\geq  a ) = 1/2$, then the argument remains unchanged. Otherwise, it suffices to generalize Lemma \ref{lem:series} with a thinning argument: assume that $a>0$ is such that one has $\PP(|\chi|\geq a) =p >0$. Consider a sequence $(B_j)_{j \in \N}$ of independent  Bernoulli random variables of parameter $p$. If $(\varepsilon_j)_{j \in \N}$ is a sequence of independent  Bernoulli random variables of parameter $\frac{1}{2p}$ independent of $(B_j)_{j \in \N}$, then $(\varepsilon_j B_j)_{j \in \N}$ is a sequence of independent Bernoulli random variables of parameter $\frac{1}{2}$, hence $\sum_{j \in \N} \varepsilon_j B_j \omega_j = + \infty$. We get therefore that $\sum_{j \in \N}  B_j \omega_j = + \infty$ since $\varepsilon_j \in \{0, 1\}$. 
  \ER

Proposition \ref{propantimarpis} does not provide a necessary and sufficient condition on the sequence $(\omega_j)_{j \in \N}$ to guarantee the boundedness of the wavelet Gaussian randomization. The next proposition shows however that the general assumption \eqref{racj} cannot be improved stronger than
\BE \label{loglogj}
\sum_{j \in \N} \frac{\sqrt{j}}{\log(\log(j))} \omega_j <\infty
\EE
when the wavelets are compactly supported.

\begin{prop}\label{prop:haar}
 Let $(\psi^{(i)}_{j,k})$ be a basis of compactly supported wavelets. There exist a non-negative sequence $(\omega_j)_{j \in \N}$ satisfying (\ref{loglogj}) and a function $f = \sum_{i,j,k} c^{i}_{j,k} \psi^{(i)}_{j,k}$ with \[
\sup_{i,k} | c^{i}_{j,k} |  = \ome_j  , \quad \forall j \in \N \] 
  such that the wavelet Gaussian randomization $X_f$ of $f$ is a.s.  nowhere locally bounded. 
\end{prop}

\begin{proof}
We  assume from now on   that $d=1$, the case $d >1$ being similar. 
Recall that if $\chi$ is a standard Gaussian random variable, then 
$$
{\mathbb{P}}(\chi \ge \sqrt{j} ) = \frac{1}{\sqrt{2 \pi j}} e^{-j/2}(1+o(1)) \ge e^{-3j/4}
$$
for $j$ large enough. Similarly to proof of Proposition \ref{prop1}, we denote by $K$ a dyadic interval such that $\psi \ge C $ on $K$ for some $C>0$. For $j \in \NN$ and $k \in \{0, \dots, 2^j -1\}$,  we define 
$$\lambda_{j,k} := \text{Supp } \psi_{j,k} \quad \text{and} \quad K_{j,k} := \{x \in \RR : 2^jx-k \in K \}.$$ 
%For any fixed $j \ge 0$, we set $j'= \left(\lfloor \frac{2d}{d-3/4}  \rfloor +1\right)j$ and we consider the event 
Since the wavelets are compactly supported, there is $k_0 \in \N$ such that for all $j \ge 0, k, k' \in \{0,..., 2^j-1\}$, one has $\lambda_{j,k} \cap \lambda_{j,k'} = \emptyset$ if $\vert k - k' \vert \ge k_0$. 

We put  $j_n = (\lfloor\frac{2}{\log (2) - 1/4} \rfloor +1)^n$ and we define the sequence $(\omega_j)_{j \in \N}$ by setting $$\omega_j = \begin{cases}
   (\sqrt{j} \log j \log(\log j))^{-1} & \text{if } j=j_n\\
   0 & \text{otherwise.}
\end{cases}
$$
This sequence satisfies (\ref{loglogj}) but not (\ref{racj}). We consider $f$ given by
\BE \label{exhaar}
f = \sum_{n\in \N} \sum_{k \in k_0 \N, \, k \le 2^{j_n} -1} \omega_{j_n} \psi_{j_n,k} 
\EE
and its randomization
\BE
X_f=\sum_{n\in \N} \sum_{ k \in k_0 \N, \, k \le 2^{j_n} -1}  \chi_{j_n,k_n} \omega_{j_n} \psi_{j_n,k} .
\EE
We will check that  $X_f$ is not bounded on any dyadic interval $\lambda \subset [0,1]$. As in the proof of Proposition \ref{contreex}, it suffices to work on $[0,1]$.

For any $j \in \N$, we set $j'=(\lfloor\frac{2}{\log (2) - 1/4} \rfloor +1)j$ and we consider the event
$$
A_j = \left\{ \exists k \in \{0, \dots, 2^j -1\}\text{ such that } \forall k' \in k_0 \N \text { with } \lambda_{j',k'} \subset K_{j,k} \text{ one has } \chi_{j',k'} \le \sqrt{\frac{j'}{3}} \right\}.
$$
Note that there is $c = c(k_0, K) \in \N$ independent of $j$ and $j'$ such that $\# \{k \in k_0 \N: \lambda_{j',k'} \subset K_{j,k} \}= 2^{j'-j-c}$.
By the independence of the Gaussian random variables, we obtain 
\begin{eqnarray*}
\mathbb{P}(A_j) & \le & \sum_{k=0}^{2^j-1} \mathbb{P} \left(\forall k'\in k_0 \N \text { with } \lambda_{j',k'} \subset K_{j,k} \text{ one has }  \chi_{j',k'} \le \sqrt{\frac{j'}{3}}\right)\\
& \le & 2^{j}(1-e^{-j'/4})^{2^{(j'-j-c)}} \\
& \le & 2^{j} \exp (-2^{j'-j-c} e^{-j'/4})\\
& \le & 2^j \exp(-\tfrac{1}{2^c} 2^{-j}e^{j'(\log(2) - 1/4)})\\
& \le & 2^j \exp(-\tfrac{1}{2^c} 2^{-j}e^{2j})
\end{eqnarray*}
for every $j$ large enough. 
Hence, for the  $j_n = (\lfloor\frac{2}{\log (2) - 1/4} \rfloor +1)^n$, the Borel-Cantelli Lemma gives almost surely the existence of a $N \in \mathbb{N}$ and a sequence $(k_n)_{n \geq N}, \, k_n \in k_0 \N$ such that 
\begin{equation}\label{eq:bc}
    \lambda_{j_n,k_n} \subset K_{j_{n-1},k_{n-1}} \quad \text{ and for all }n \ge N , \quad \chi_{j_n,k_n} > \sqrt{\frac{j_n}{3}}.
\end{equation}

 We work now in the event  of probability $1$ given by the Borel-Cantelli Lemma.  The term $$f_N = \sum_{n < N} \sum_{k \in k_0 \N, \, k \le 2^{j_n} -1}  \chi_{j_n,k} \omega_{j_n}\psi_{j_n,k}$$ is obviously bounded and we focus on the rest. For any integer $M \ge N+1$, one has on $\lambda_{j_M,k_M}$
\begin{eqnarray*}
 g_N &: =& \sum_{n =N}^{+\infty} \sum_{k \in k_0 \N, \, k \le 2^{j_n} -1}  \chi_{j_n,k} \omega_{j_n} h_{j_n,k}\\
 & = & \sum_{n =N}^{M-1}  \chi_{j_n,k_n} \omega_{j_n} h_{j_n,k_n}  + \sum_{n \ge M} \sum_{k \in k_0 \N: \, \lambda_{j_n,k} \subset \lambda_{j_M,k_M}}  \chi_{j_n,k} \omega_{j_n} h_{j_n,k} \\
 &\ge & \sum_{n=N}^{M-1}  C\sqrt{\frac{j_n}{3}} \omega_{j_n} + \sum_{n \ge M} \sum_{k \in k_0 \N: \, \lambda_{j_n,k} \subset \lambda_{j_M,k_M}}  \chi_{j_n,k} \omega_{j_n} \psi_{j_n,k}.
\end{eqnarray*}
The second term has a vanishing integral on  $\lambda_{j_M,k_M}$ and it follows that the average value of $g_N$ on  $\lambda_{j_M,k_M}$ is larger than $C \sum_{n=N}^M \sqrt{\frac{j_n}{3}}\omega_{j_n}$. Since this term is not bounded in $M$, we get the conclusion.
\end{proof}

\BR \label{rem:chaos2}
\begin{enumerate}
    \item 

Again, this result can be extended to more general laws satisfying a condition of the type
$$
   \PP(|\chi|\geq x) \geq C' e^{-a x^\gamma} 
$$
for all $x>0$ large enough and 
for some constant $C',a>0$ and some exponent $\gamma \in (0,2]$. Note that is it for example the case of random variables living in a Wiener chaos of order $n \geq 1$ with $\gamma=\frac{2}{n}$, see \cite[Theorem 6.12]{janson}. In this case,  \eqref{loglogj} is replaced by $$
\sum_{j \in \N} \frac{j^{\frac{1}{\gamma}}}{\log(\log(j))} \omega_j <\infty .
$$
\item Conditions of type $2^{js}j^{\gamma}\omega_j<+\infty$ for $s\ge0$ means that the function $f$ belongs to the periodic Besov space of logarithmic smoothness $B^{s,\gamma}_{p,q}(\TT)$ defined via the Fourier-analytical approach.
The assumption $\sqrt{j}\omega_j <\infty$ of Proposition \ref{propantimarpis}  is equivalent to the fact that $f$ belongs to $B^{0,1/2}_{\infty, \infty}(\TT)$ and hence that $f$ is the derivative of a function of $B^{1,1/2}_{\infty,\infty}(\TT)$. This last space can be identified to the Zygmund class with logarithmic smoothness ${\cal{Z}}^{1,1/2}(\TT)$ :
$$
{\cal{Z}}^{1,1/2}(\TT) := \{ f \in {\cal{S}}'(\TT): \, \sup_{0< \vert h \vert <1} \frac{ \vert f(x+h)+f(x-h)- 2 f(x) \vert \sqrt{1-\log(h)}}{\vert h \vert} <+\infty\}.
$$
Note that for $s>0$, the definition of Besov spaces of logarithmic smoothness {\it via} the Fourier-analytical approach coincides with the definition of classical Besov spaces of logarithmic smoothness (defined {\it via} the integral on iterated differences of $f$). The classical definition can be extended to $s=0$ but does no more coincide with the Fourier-analytical approach. We refer to \cite{DT} for the identification of $B^{1,1/2}_{\infty,\infty}(\TT)$ with $ {\cal{Z}}^{1,1/2}(\TT) $ and for an extensive study of function spaces of logarithmic smoothness and to \cite{Almeida} for their wavelet characterizations.
\end{enumerate}
\ER

%\begin{tempo} 
 %introduire les espaces $L^{\infty}_{log}$ definis comme les derives au sens des distributions des fonction de la classe de Zygmund avec correction logarithmique et obtenuir leur caracterisation ondelettes: $| C_{j,k} | \leq C \sqrt{j}$ equivaut a ce que la primitive verifie 
%$ | F(x) -F(y) | \leq C | x-y| \sqrt{| \log (|x-y|)| }$ 

%Par l'argument du module de continuite, on a $C^1_{\log} = B^1_{\infty,\infty,\log}$ (meme caracterisation en ondelettes). Dans \text{https://arxiv.org/pdf/1811.06399.pdf} on a le resultat de derivation qu'il faut. Par contre, un resultat interessant : les espaces de Besov definis en Fourier ou definis a l'aide de la formule integrale coincident pour $s>0$ mais pas pour $s=0$ : il y a une correction log d'exposant 1/2 qui apparait.
%\end{tempo}

  \section{Modulus of continuity and concluding remarks} 
  
  \label{seccont}

One might expect that the loss of continuity (or boundedness)  through a  randomization of wavelet coefficients follows from the fact that continuity and boundedness are notions which cannot be characterized by a condition on the moduli of the wavelet coefficients. We now show that it is not the case, by considering the  H\"older spaces $C^\al $ for which the multiplier property holds.
%Recall that, since theses spaces are not separable, they cannot have unconditional bases; however they have unconditional weak$^*$ bases (this means that the requirement of strong convergence of the partial sums of the reconstruction is replaced by a   weak$^*$  convergence, see  \cite{}
The  spaces $C^\al (\RR^d) $ are characterized by a simple condition on the wavelet coefficients, see \cite{Mey90I}. Recall first that if $ \al \in (0,1)$, then $f \in C^\al (\RR^d)$ if $f \in L^\infty(\R^d)$ and if
\[ \exists C  >0 \quad \text{ such that } \quad | f(x)-f(y) | \leq C | x-y|^\al   \quad \forall x,y\in \R^d\]
and the same definition also applies in the periodic case. We also refer to \cite{Mey90I} for extensions when $\al \notin (0,1)$.  If the wavelet is $1$-smooth, the following wavelet characterization of the H\"older spaces holds: a function $f$ belongs to  $C^\al (\RR^d) $  if and only if  its wavelet coefficients satisfy 
\BE  \label{caractholder} \exists C>0  \quad \text{ such that } \quad \begin{cases}
|c_k| \leq C & \forall k \in \ZZ^d \\[1.5ex]
 |c^{i}_{j,k}| \leq C 2^{-\al j}& \forall  i\in \{1,\dots, 2^d-1\}, j\in \N ,k \in \ZZ^d 
 \end{cases}\EE
 and the characterization also holds in the periodic case. 
We will see that the H\"older spaces supply an interesting setting where the  conclusion concerning Rademacher and Gaussian randomization strongly differ. 

We start by some simple remarks.  First, of course, \eqref{caractholder} implies, both in the periodic and non-periodic case,  that a function in $C^\al $ remains in $C^\al $ after an i.i.d. randomization of the wavelet coefficients with $\chi$ bounded.

As regards  randomization using an unbounded random variable, one has to separate the periodic and non-periodic settings. 

In the case of functions defined on $\RR^d$, unbounded randomization  has the effect that, in general, the randomized wavelet series of a function in a $C^\al(\RR^d)$ space does not belong to $L^\infty(\RR^d)$. Consider for instance the function defined by the wavelet series
\BE \label{ftcal} \sum_{j \in \N,i,k} 2^{-\al j} \Psi^{(i)}_{j,k} . \EE
Using \eqref{caractholder}, this function clearly belongs to $C^\al (\RR^d) $. However, if  $\chi$ is unbounded, then  its randomization
\[ \sum_{j \geq 0,i,k} 2^{-\al j} \chi^i_{j,k} \Psi^{(i)}_{j,k}  \]
does not belong to $L^\infty(\RR^d)$. Indeed, at a given scale $j$, there is  an infinite number of  random variables $\chi^i_{j,k}$ so that  one has
\[ \sup_{i,k} |  \chi^i_{j,k} | = + \infty. \]
We can however expect that this stochastic  process is locally bounded. This is actually equivalent to considering the periodic setting, which we now do.  More precisely, we will study the uniform regularity of the randomization via its modulus of continuity (this context includes the case of the H\"oler spaces). We recall 
that a {modulus of continuity} is an increasing function $\theta$: $\RR^+ \rightarrow  \RR^+$ satisfying 
 $\theta (0) = 0$  and such that $\theta (2x) \leq C \theta (x)$ for all $x \in \RR^+$ and a constant $C>0$. 
We say that $\theta$ is a modulus of continuity of $f$ if there exists a constant $C >0$ such that 
  \BE \label{eq:modulus} \left| f(x) -f(y) \right| \leq C \theta (|x-y|) \EE
  for every $x,y$. Wavelet characterizations of regularity require the  following additional regularity property for moduli of continuity, see \cite{JaffMey1}. 
 
 \BD A modulus of continuity $\theta$ is regular if 
$$ \exists  N \in \N, \quad \forall  J \in \N, \qquad \label{modreg}   \begin{cases}  \displaystyle \sum_{j=J}^\infty 2^{Nj} \theta (2^{-j})  \leq C 2^{NJ} \theta (2^{-J}) \\[2ex] \displaystyle \sum_{j=-\infty}^J 2^{(N+1)j} \theta (2^{-j})  \leq C 2^{(N+1)J} \theta (2^{-J}).     \end{cases}  $$
 \ED

\BP\label{prop:caractunif}
Let $\theta$ be a  modulus of continuity and let $f:\RR^{d}\to
\R$ be a function. % and let $x_{0}\in \R$.
% \begin{enumerate}
% \item If $\theta$ is a pointwise modulus of continuity of $f$ at
%   $x_{0}$, then 
% \BE \label{eq:charactponct}
% \exists C>0 \quad \forall i,j,k \quad |\cjk| \leq C \left(
%   \theta(2^{-j})+ \theta \Big(\big| x_{0}-\frac{k}{2^j}\big|\Big)\right).
% \EE
%Conversely, if \eqref{eq:charactponct} holds
If   $\theta$ is a  modulus of continuity of $f$, then
\BE\label{eq:charactunif}
\exists C>0 \quad \forall i,j,k \quad |c^i_{j,k}| \leq C 
  \theta(2^{-j}).
\EE
Conversely, if \eqref{eq:charactunif} holds and if $\theta$ is
regular, then $\theta$ is a  modulus of continuity of $f$.
% \end{enumerate}
\EP 

Together with Remark 
\ref{rem:chaos1}, this characterization of uniform moduli of continuity
 directly gives the following result.

\BP \label{unifmod1} 
 Let $\chi^{(i)}_{j,k}$ be  a  sequence of  I.I.D. random variables with common law $\chi$ with exponential decay  of order $\gamma$ (see \eqref{eq:expdecay}) and let $f \in C^\alpha(\TT)$ where $\alpha \in (0,1)$.  
  Then the wavelet randomization $X_f$ of $f $ satisfies that almost surely, there exists $C>0$ such that
$$ \forall x, y \in \TT \qquad 
|X_f(x) - X_f(y)|\leq C |x-y|^\alpha |\log(|x-y|)|^{\frac{1}{\gamma}} .
$$

 \EP

The optimality of Proposition \ref{unifmod1} can be obtained for laws already considered in Remark \ref{rem:chaos2}, i.e.  random variables belonging to a Wiener chaos of order $\frac{2}{\gamma}$ and in particular for standard Gaussian random variables ($\gamma = 2$).

\BP \label{prop:optimalmodulus}
%Assume that the wavelet $\psi$ has fast decay. 
 Let $\chi^{(i)}_{j,k}$ be  a  sequence of  I.I.D. random variables with common law $\chi$ and assume that there exists some constant $C,C',a,b>0$ and an exponent $\gamma \in (0,2]$ such that, for  $x>0$ large enough,
 $$  
C' e^{-a x^{\gamma}} \leq \PP(|\chi|\geq x) \leq C e^{-b x^\gamma} . $$
 Let us consider the function $f$ defined on $\TT$ by
  $$
  f = \sum_{j \in \N} \sum_{i,k} 2^{-\alpha j} {\psi}^{(i)}_{j,k}
  $$
  where $\alpha \in (0,1)$. 
Then, almost surely the wavelet randomization $X_f$ of $f$ satisfies 
  $$
 0< \sup_{x\neq y} \frac{|X_f(x) - X_f(y)|}{|x-y|^\alpha |\log(|x-y|)|^{\frac{1}{\gamma}} }< + \infty.
  $$
\EP

\begin{proof}
The upper bound has already been obtained in Proposition \ref{unifmod1}. Note that in particular, $X_f$ is almost surely bounded. For the lower bound, we proceed as in 
Proposition \ref{prop:haar} and Remark \ref{rem:chaos2} to show that  for any  $c \in (0,{d \log 2}/{a})$, almost surely, there exist infinitely many scales $j$ such that 
$$
\sup_{i,k} |\chi^{(i)}_{j,k}| \geq (c j)^{\frac{1}{\gamma}} . 
$$
 The orthogonality of the wavelets and their first vanishing moment allow then to write
\begin{eqnarray}\label{eq:modulusBIS}
(c j)^{\frac{1}{\gamma}}2^{-\alpha j} \leq |\chi^{(i)}_{j,k}| 2^{-\alpha j} 
& \leq &2^{dj} \int_{\R^d} |X_f(x) - X_f(k2^{-j})| |\psi(2^jx-k)| dx \nonumber \\
& = &  \int_{\R^d} \big|X_f((u+k)2^{-j}) - X_f(k2^{-j})\big| |\psi(u)| du
\end{eqnarray}
We decompose now the integral into two parts according to $|u|\leq 2^{-j}$ or $|u|> 2^{-j}$. For the fist case, we write
\begin{eqnarray*}
 &&  \int_{|u|\leq 2^{-j}} \big|X_f((u+k)2^{-j}) - X_f(k2^{-j})\big| |\psi(u)| du  \\
   & \leq & 2^{-\alpha j}  \sup_{x \neq y}  \frac{|X_f(x) - X_f(y)|}{|x-y|^\alpha |\log(|x-y|)|^{\frac{1}{\gamma}} } \int_{|u| \leq 2^{-j}} |u|^\alpha |\log(|u2^{-j}|)|^{\frac{1}{\gamma}} |\psi(u)| du\\
   & \leq & D 2^{-\alpha j} j^{\frac{1}{\gamma}}  \sup_{x \neq y}  \frac{|X_f(x) - X_f(y)|}{|x-y|^\alpha |\log(|x-y|)|^{\frac{1}{\gamma}} }
\end{eqnarray*}
for some constant $D>0$, by noticing that $|\log(|u2^{-j}|)|^{\frac{1}{\gamma}}$ can be bounded  by $ |\log(|u|)|^{\frac{1}{\gamma}}|\log(|2^{-j}|)|^{\frac{1}{\gamma}} $ and using the decay of the wavelet. For the second term, we use the boundedness of $X_f$ together with the fast decay of the wavelet to obtain
\begin{eqnarray*}
 \int_{|u|> 2^{-j}} \big|X_f((u+k)2^{-j}) - X_f(k2^{-j})\big| |\psi(u)| du  
   & \leq  & 2 \|X_f\|_{\infty}\int_{|u|> 2^{-j}}  \frac{1}{(1+|u|)^{N+2}} du\\& 
 \leq & D' 2^{-Nj}
\end{eqnarray*}
for $N>\alpha$ and a constant $D'>0$. 
Putting these two cases together with \eqref{eq:modulusBIS} allows to write
$$
(c j)^{\frac{1}{\gamma}}2^{-\alpha j} \leq D 2^{-\alpha j} j^{\frac{1}{\gamma}}  \sup_{x \neq y}  \frac{|X_f(x) - X_f(y)|}{|x-y|^\alpha |\log(|x-y|)|^{\frac{1}{\gamma}} } + D' 2^{-Nj}
$$
which gives the conclusion. 
\end{proof}

\BR
The assumption of orthogonality of the wavelets can be weakened by imposing that the mother wavelet generates a biorthogonal basis. This setting allows to consider for example the fractional Brownian motion, by using the remarkable decomposition of
Meyer, Sellan and Taqqu \cite{MST}: Indeed, the fractional Brownian motion $B_{H}$ of Hurst exponent $H \in (0,1)$ can be written as 
\begin{equation}\label{eq:mbf}
B_{H}(t) = \sum_{j \in \N}\sum_{k \in \ZZ} 2^{-Hj} \xi_{j,k} \psi_{H+1/2} (2^{j}t-k)
 + R(t)
\end{equation}
where $R$ is a smooth process, $(\xi_{j,k})_{(j,k)\in\N  \times \ZZ}$ is a
sequence of I.I.D. standard Gaussian random variables, and  $\psi_{h+1/2} $ generates a biorthogonal wavelet basis. The optimality of the modulus of continuity  given in Proposition  \ref{prop:optimalmodulus} for of the Gaussian randomization a function in $C^{\alpha}(\TT)$ has been obtained using a function $f$ whose randomization is, up to a smooth process, very similar to the fractional Brownian motion. The precise pointwise behavior of such random wavelet series has been deeply studied in \cite{EsserLoosveldt}. In particular the authors obtained that there are two other kinds of points at which the pointwise oscillation is slower than the one given by the modulus of continuity. See also \cite{DawLoosveldt} for results in a chaos of order 2. 
\ER

If the random variables $\xi_{j,k}$ satisfy only the fast decay assumption given in \eqref{eq:fastdecay}, one cannot obtain the exact modulus of continuity. However, if $f\in C^{\alpha}(\TT)$, then almost surely, for all $\varepsilon>0$, one has $X_f \in C^{\alpha -\varepsilon}$.  
Hence, if one defines the {uniform H\"older exponent} of a function $f$ on $\TT$  by
  \[ H_{\min} (f) = \sup \{ \al>0 :  f \in C^\al (\TT) \} , \]  
 then one has
  \[ H_{\min} (X_f) \geq H_{\min} (f) \] 
almost surely. The inequality above  actually is an equality: indeed let us fix $\beta >0$ such that $f \notin C^{\beta}(\TT)$. The wavelet characterization of the H\"older space gives the existence of a sequence $(i_{n},j_{n},k_{n})$ such
that the wavelet coefficients of $f$ satisfy $
\sup_{n \in \N} 2^{\beta j_{n}} |c^{i_{n}}_{j_{n},k_{n}}| =  +
\infty $.
For every $\varepsilon>0$,  
$$
\sum_{n \in \N}\mathbb{P}(|\chi^{i_n}_{j_{n},k_{n}}| \geq 2^{- \varepsilon
  j_{n}}) = + \infty ,
$$
and  the Borel-Cantelli Lemma implies that  $\sup_{n\in \N} 2^{(\beta + \varepsilon)
  j_{n}} |c^{i_n}_{j_{n},k_{n}}|  |\chi_{j_{n},k_{n}}| =  +
\infty $, hence $X \notin C^{\beta + \varepsilon}(\TT)$ almost surely. It follows that
  \[ H_{\min} (X_f) = H_{\min} (f)\] 
almost surely.

 Let us conclude  by working out an example which puts into light in a concrete way the  fundamental  difference between   Fourier  vs.  wavelet randomization. 
  Let $\parx$ be the ``sawtooth function''
\[ \parx   = 
\begin{cases}
x-[ x] -\frac{1}{2} & \mbox{if} \;\; x \notin \ZZ \\
 & 0  \mbox{ otherwise.} \end{cases} \]
Its Fourier series is 
\BE \label{poisson} \parx = -\sum_{ m=1}^\infty \frac{ \sin ( 2 \pi mx)}{\pi m}. \EE
% Our purpose is to compare on this example the  randomisation  on Fourier coefficients vs. the wavelet randomisation. 
Recall  now Wiener's remarkable Fourier expansion for the Brownian motion on $[0,1]$:
\[ B(x) =  \sqrt{2}  \chi_0  x +  \sum_{ m=1}^\infty  \chi_m \frac{ \sin (2  \pi mx)}{ \pi m}, \]
where  $(\chi_m)_{m \in \N} $ is a sequence of I.I.D. standard   Gaussian  random variables. 
Therefore the Gaussian randomization of the sawtooth function is, up to a linear term, the Brownian motion. This example makes explicit the smoothing effect of this randomization: we start with a discontinuous function, and we obtain  a stochastic process the sample paths of which belong to $C^{1/2 -\ep }$ for any $\ep >0$; but this gain in regularity is compensated by the fact that smoother singularities are ``spread everywhere'':  almost surely, for all  $x_0$, one has $ B$ does not belong to the pointwise H\"older space $C^{1/2} (x_0) $.   

Assume now that we use compactly supported wavelets. For  $j$  large enough  the support of $\pjk$  has length less than 1,  and the periodized wavelets coincide with the wavelets on $\RR$; it follows that the wavelet coefficients $c_{j,k}$ of the sawtooth function are the same as those of  the Heaviside function \eqref{heavi1} (up to a factor 2). Therefore equation \eqref{heavi2} holds, thus yielding an infinite number of wavelet coefficients with exactly the same size. Since Gaussian random variable  can take arbitrarily large values, it follows that, after randomization this sequence is not bounded, and  Proposition \eqref{propcritqsimple} yields that almost surely  the sample paths of the corresponding process are not bounded. 

\begin{figure*}[h!]
\begin{center}
\includegraphics[scale=0.25]{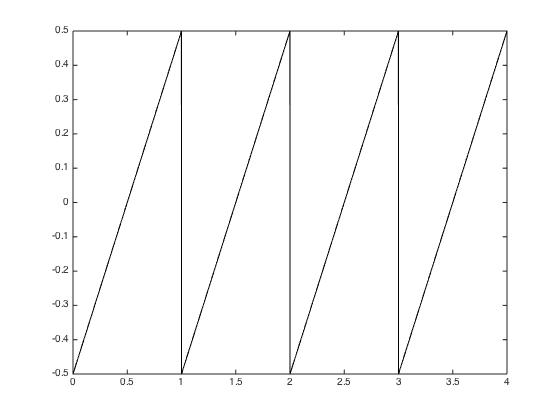}
\includegraphics[scale=0.25]{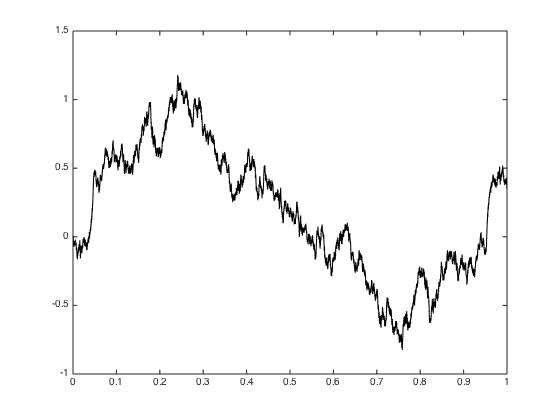}
\includegraphics[scale=0.25]{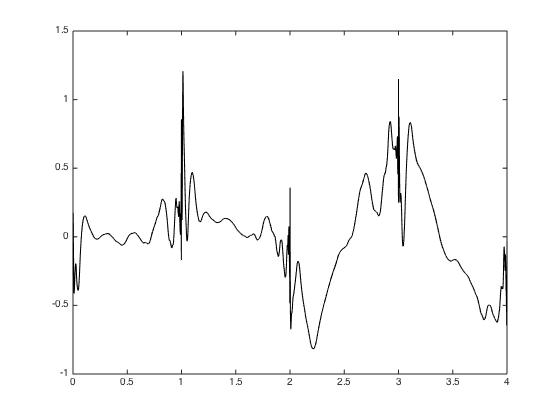}
\caption{From left to right : The Sawtooth function, approximation of the Brownian motion obtained with the help of the Randomized Fourier Series of the Sawtooth function (Wiener Formula), and the Randomized Wavelet Series of the Sawtooth function 
 with Daubechies-10 wavelets. Signals are of length $2^{17}$ and the approximation of the Brownian motion is stopped at $m=2^{14}$ for the Sine Series. One can observe that the randomized wavelet series starts to diverge at the discontinuities of the sawtooth function and is smooth elsewhere, in sharp contradistinction with the Fourier randomization which cannot keep memory of the localization of the singularities. }
\end{center}
\end{figure*}

Note however that, for wavelet randomization, the singular behavior remains located  where the singular behavior of the initial function was: for any $\ep >0$,  the stochastic process thus obtained  is bounded, and even is as smooth that the wavelet of any interval $\mathbb{T} \setminus [-\ep, \ep] $. This is in sharp contradistinction with Fourier randomization. 
Let us already mention that the study of the pointwise regularity of random wavelet series and random Fourier series  will be investigated in details in a forthcoming work.

 \bibliographystyle{plain}
 \bibliography{bib-Gauss.bib}

\end{document}